\documentclass[a4paper]{scrartcl}
\usepackage[T1]{fontenc}
\usepackage{amsmath,bbm}
\usepackage[english]{babel}
 \usepackage{amsthm}
\usepackage{enumerate}
\usepackage[usenames,dvipsnames]{color}
\usepackage[all]{xy}
\usepackage{verbatim}

\usepackage[charter]{mathdesign}

\newcommand{\N}{\mathbb{N}}
\newcommand{\Z}{\mathbb{Z}}
\newcommand{\Q}{\mathbb{Q}}

\newcommand{\OreGen}{R[x;\sigma,\delta]}
\newcommand{\SkewPol}{R[x;\sigma,0]}
\newcommand{\DiffPol}{R[x;\identity_R,\delta]}

\DeclareMathOperator{\identity}{id}

\DeclareMathOperator{\CHAR}{char}

\theoremstyle{plain}
\newtheorem{theorem}{Theorem}[section]
\newtheorem{lemma}[theorem]{Lemma}
\newtheorem{prop}[theorem]{Proposition}
\newtheorem{corollary}[theorem]{Corollary}
\theoremstyle{definition}
\newtheorem{definition}[theorem]{Definition}
\newtheorem{exmp}[theorem]{Example}

\newtheorem{remark}[theorem]{Remark}
\newtheorem{question}{Question}

\begin{document}
\title{Maximal commutative subrings and simplicity of Ore extensions}

\author{Johan {\"O}inert\footnote{Department of Mathematical Sciences, University of Copenhagen,  Universitetsparken 5, DK-2100 Copenhagen \O, Denmark, E-mail: oinert@math.ku.dk} \and
 Johan Richter\footnote{Centre for Mathematical Sciences, Lund University, Box 118, SE-22100 Lund, Sweden, E-mail: johanr@maths.lth.se} \and Sergei D. Silvestrov\footnote{Division of Applied Mathematics,  The School of Education, 
Culture and Communication,
M{\"a}lardalen University, Box 883, SE-72123 V{\"a}ster{\aa}s, Sweden, E-mail: sergei.silvestrov@mdh.se}}

\date{}

\maketitle

\begin{abstract}

The aim of this article is to describe necessary and sufficient conditions for simplicity of Ore extension rings, with an emphasis on differential polynomial rings. We show that a differential polynomial ring,
 $R[x;\identity_R,\delta]$, is simple if and only if 
its center is a field and $R$ is $\delta$-simple. When $R$ is commutative we note that the centralizer of $R$ in $\OreGen$ is a maximal commutative subring containing $R$ and, in the case when $\sigma=\identity_R$,
 we show that it intersects every non-zero ideal of $\DiffPol$ non-trivially. Using this we show that if $R$ is $\delta$-simple and maximal commutative in $\DiffPol$, then $\DiffPol$ is simple. We also show that under
some conditions on $R$ the converse holds.\\

\end{abstract}


\section{Introduction}

A topic of interest in the field of operator algebras is the connection between properties of dynamical systems and algebraic properties of crossed products associated with them. More specifically the question when  
a certain canonical subalgebra is maximal commutative and has the ideal intersection property, i.e. each non-zero ideal of the algebra intersects the subalgebra non-trivially. For a topological dynamical systems $(X,\alpha)$ 
one may define a crossed product C*-algebra $C(X) \rtimes_{\tilde{\alpha}} \Z $ where 
$\tilde{\alpha}$ is an automorphism of $C(X)$ induced by $\alpha$. It turns out that the property known as topological freeness of the dynamical system is equivalent to $C(X)$ being a maximal commutative subalgebra of 
$C(X) \rtimes_{\tilde{\alpha}} \Z$ and also equivalent to the condition that every non-trivial closed ideal has a non-zero intersection with $C(X)$. An excellent reference for this correspondence is \cite{TomiyamaBook}.
For analogues, extensions and applications of this theory in the study of dynamical systems, harmonic analysis, quantum field theory, string theory, integrable systems, fractals and wavelets see \cite{ArchbordSpielberg, 
CarlsenSS,DutkayJorg3,DutkayJorg4,MR2195591ExelVershik,MACbook3,
OstSam-book,TomiyamaBook}.

For any class of graded rings, including gradings given by semigroups or even filtered rings (e.g. Ore extensions), it makes sense to ask whether the ideal intersection property is related
to maximal commutativity of the degree zero component. 
For crossed product-like structures, where one has a natural action, it further makes sense to ask how the above mentioned properties of the degree zero component are related to properties of the action.

These questions have been considered recently for algebraic crossed products and Banach algebra crossed products, both in the traditional context of crossed products by groups as well as generalizations
to graded rings, crossed products by groupoids and general categories in
\cite{deJeu,LundstromOinert,OinertSimplegroupgradedrings,OinertLundstrom,OinertLundstrom2,OinertSS1,OinertSS2,OinertSS3,OinertSS4,OinertSS5,SSD1,SSD2,SSD3,SvT}.

Ore extensions constitute an important class of rings, appearing in extensions of differential calculus, in non-commutative geometry, in quantum groups and algebras and as a uniting framework for many algebras
appearing in physics and engineering models. An Ore extension of $R$ is an overring with a generator $x$ satisfying $xr=\sigma(r)x+\delta(r)$ for $r \in R$ for some endomorphism $\sigma$ and a 
 $\sigma$-derivation $\delta$. 

This article aims at studying the centralizer of the coefficient subring of an Ore extension, investigating conditions for the simplicity of Ore extensions and demonstrating the connections between these two topics.

Necessary and sufficient conditions for a differential polynomial ring (an Ore extension with $\sigma = \identity_R$) to be simple have been studied before. An early paper by Jacobson \cite{Jacobson} studies the case when
$R$ is a division ring of characteristic zero. His results are generalized in the textbook \cite[Chapter 3]{CozzensFaith} in which Cozzens and Faith prove that if $R$ is a 
$\Q$-algebra and $\delta$ a derivation on $R$ then $R[x;\identity_R, \delta]$ is simple if and only if $\delta$ is a so called outer derivation and the only ideals invariant under $\delta$ are $\{0\}$ and $R$ itself.  In his
PhD thesis \cite{Jordan} Jordan shows that
if $R$ is a ring of characteristic zero and with a derivation $\delta$ then $R[x;\identity_R,\delta]$ is simple if and only if $R$ has no non-trivial $\delta$-invariant ideals and $\delta$ is an 
outer derivation. In $\cite{Jordan}$ Jordan also shows that if $\DiffPol$ is simple, then $R$ has zero or prime characteristic and gives necessary and sufficient conditions for $\DiffPol$ to be simple when $R$ has prime 
characteristic.
(See also \cite{Jordan2}.)

In \cite{CozzensFaith} Cozzens and Faith
also prove that if $R$ is an commutative domain, then $R[x;\identity_R,\delta]$ is simple if and only if the subring of constants, $K$, is a field (the constants are the elements in the kernel of the derivation) and $R$ is 
infinite-dimensional as a vector space over $K$. In \cite[Theorem 2.3]{Goodearl82} Goodearl and Warfield prove that if $R$ is a commutative ring and $\delta$ a derivation on $R$, then $R[x;\identity_R,\delta]$ is simple if and only if there are
no non-trivial $\delta$-invariant ideals (implying that the ring of constants, $K$, is a field) and $R$ is infinite-dimensional as a vector space over $K$.
  
McConnell and Sweedler \cite{McConnell} study simplicity criteria for smash products, a generalization of differential polynomial rings. 

Conditions for a general Ore extension to be simple have been studied in \cite{LamLeroy} by Lam and Leroy. Their Theorem 5.8 says that $S=R[x;\sigma,\delta]$ is non-simple if and only if 
there is some $R[y;\sigma',0]$ that can be embedded 
in $S$. 
See also \cite[Theorem 4.5]{LamLeroy} and \cite[Lemma 4.1]{JainLamLeroy} for necessary and sufficient conditions for $R[x;\sigma,\delta]$ to be simple. In \cite[Chapter 3]{CozzensFaith} Cozzens and Faith provide an example of
 a simple Ore extension $R[x;\sigma,\delta]$ which is not a differential polynomial ring.  
 
If one has a family of commuting derivations, $\delta_1,\ldots, \delta_n$, one can form a differential polynomial ring in several variables. The articles \cite{Malm, Posner, Voskoglou} consider the question when such 
rings are simple. In 
\cite{Hauger} a class of rings with a definition similar, but not identical to, the defintion of differential polynomial rings of this article, are studied and a characterization of when they are simple is obtained. 

None of the articles cited have studied the simplicity of Ore extensions from the perspective pursued in this article. In particular for differential polynomial rings the connection between maximal commutativity of the
coefficient subring and simplicity of the differential polynomial ring (Theorem \ref{MainResult}) appears to be new,
 as well as the result that the centralizer of the center of the coefficient subring has the ideal intersection property (Proposition \ref{IdealIntersection}).  
We also show that a differential ring is simple if and only if its center is a field and the coefficient subring has no non-trivial ideals invariant under the derivation (Theorem \ref{MainResult2}). In Theorem \ref{SimpleOre} we
note that simple Ore extensions over commutative domains are necessarily differential polynomial rings, and hence can be treated by the preceding characterization.

In Section \ref{sec:defnotOreExt}, we recall some notation and basic facts 
about Ore extension rings used throughout the rest of the article.
In Section \ref{sec:centralizerMaxcom}, we describe the centralizer of the coefficient subring
in general Ore extension rings and then use this description to provide conditions for maximal commutativity of the coefficient subring. These conditions of maximal commutativity of the coefficient subring
are further detailed for two important classes of Ore extensions, the skew polynomial rings and differential polynomial rings
in Subsections \ref{subsec:skewpolringscentermaxcom} and \ref{subsec:Differentialpolringscentermaxcom}.
 In Section \ref{sec:simplicity}, we investigate when an Ore extension ring 
is simple and demonstrate how this is connected to maximal commutativity of the coefficient subring
for differential polynomial rings (Subsection  \ref{subsec:simplicityDifferentialPolrings}).

\section{Ore extensions. Definitions and notations}
\label{sec:defnotOreExt}

Throughout this paper all rings are assumed to be unital and associative, and ring morphisms are assumed to respect multiplicative identity elements.

For general references on Ore extensions, see \cite{Goodearl2004,McConnellRobson,RowenRing1}. For the convenience of the reader,
we recall the definition. Let $R$ be a ring, 
$\sigma : R \to R$ a ring endomorphism (not necessarily injective) and $\delta : R \to R$ a $\sigma$-derivation, i.e.
\begin{displaymath}
	\delta(a+b)=\delta(a)+\delta(b)
	\quad \text{ and } \quad
	\delta(ab) = \sigma(a) \delta(b) + \delta(a)b
\end{displaymath}
for all $a,b\in R$.

\begin{definition}
The Ore extension $\OreGen$ is defined as the ring generated by $R$ and an element $x \notin R$ such that $1,x, x^2, \ldots$ form a basis for $\OreGen$ as a left $R$-module and all $r \in R$ satisfy
\begin{equation}\label{MultiplicationRule}
	x r = \sigma(r)x + \delta(r).
\end{equation}
\end{definition}
Such a ring always exists and is unique up to isomorphism (see \cite{Goodearl2004}). 
Since $\sigma(1)=1$ and $\delta(1 \cdot 1) = \sigma(1) \cdot \delta(1) +\delta(1) \cdot 1$, we get that $\delta(1)=0$ and hence $1_R$ will be a multiplicative identity for $\OreGen$ as well. 

If $\sigma= \identity_R$, then we say that $\DiffPol$ is a \emph{differential polynomial ring}.
 If $\delta \equiv 0$, then we say that $\SkewPol$ is a 
\emph{skew polynomial ring}. The reader should be aware that throughout the literature on Ore extensions the terminology varies.

An arbitrary non-zero element $P \in \OreGen$ can be written uniquely
as $P = \sum_{i=0}^n a_i x^i$ for some $n \in \Z_{\geq 0}$, with $a_i \in R$ for $i \in \{0,1,\ldots, n\}$ and $a_n \neq 0$. The \emph{degree} of $P$ will be defined as $\deg(P):=n$. We set $\deg(0) := -\infty$.

\begin{definition}
 A $\sigma$-derivation $\delta$ is said to be \emph{inner} if there exists some $a \in R$ such that $\delta(r) = ar-\sigma(r)a$ for all $r \in R$. A $\sigma$-derivation that is not inner is called \emph{outer}.
\end{definition}

Given a ring $S$ we denote its center by $Z(S)$ and its characteristic by $\CHAR(S)$.
The \emph{centralizer} of a subset $T\subseteq S$ is defined as the set of elements of $S$ that commute with every element of $T$. If $T$ is a commutative subring of $S$ and the centralizer of $T$ in $S$ coincides with $T$, 
then $T$ is said to be a \emph{maximal commutative} subring of $S$.

\section{The centralizer and maximal commutativity of $R$ in $\OreGen$}
\label{sec:centralizerMaxcom}
In this section we shall describe the centralizer of $R$ in
the Ore extension $\OreGen$ and give conditions for when $R$ is a maximal commutative subring of $\OreGen$. We start by giving a general description of the centralizer and then
derive some consequences in particular cases. \\

\noindent In order to proceed we shall need to introduce some notation. We will define functions $\pi_k^l: R \to R$  for $k,l \in \Z$.
We define $\pi_0^0 = \identity_R$. If $m,n$ are non-negative integers such that $m>n$, or if at least one of $m,n$ is negative, then we define $\pi_m^n \equiv 0$. The remaining cases are defined 
by induction through the formula
\begin{displaymath}
 \pi_m^n= \sigma \circ \pi_{m-1}^{n-1} + \delta \circ \pi_m^{n-1}.
\end{displaymath}
These maps turn out to be useful when it comes to writing expressions in a compact form. We find by a straightforward induction that for all $n \in \Z_{\geq 0}$ and $r \in R$ we may write
\begin{displaymath}
 x^n r = \sum_{m=0}^{n} \pi_m^n(r) x^m.
\end{displaymath}

\begin{prop}\label{CommutantR}
 $\sum_{i=0}^n a_i x^i \in \OreGen$ belongs to the centralizer of $R$ in $\OreGen$ if and only if
\begin{displaymath}
 r a_i = \sum_{j=i}^{n} a_j \pi_i^j(r)
\end{displaymath}
holds for all $i \in \{0,\ldots,n\}$ and all $r \in R$.
\end{prop}

\begin{proof}
For an arbitrary $r\in R$ we have $r \sum_{i=0}^{n} a_i x^i = \sum_{i=0}^{n} r a_i x^i$ and
\begin{eqnarray*}
 \sum_{i=0}^{n} a_i x^i r = \sum_{i=0}^n a_i \sum_{j=0}^{i} \pi_j^i(r) x^j = \sum_{i=0}^{n} \sum_{j=0}^{n} a_i \pi_j^i(r) x^j = \\
\sum_{j=0}^n \sum_{i=0}^{n} a_i \pi_j^i(r) x^j = \sum_{i=0}^n \sum_{j=0}^n a_j \pi_i^j(r) x^i.
\end{eqnarray*}
By equating the expressions for the coefficient in front of $x^i$, for $i\in \{0,\ldots,n\}$, the desired conclusion follows.
\end{proof}

The above description of the centralizer of R holds in a completely general setting. We shall now use it to obtain conditions for when $R$ is a maximal commutative subring of the Ore extension ring.

\begin{remark}
 Note that if $R$ is commutative, then the centralizer of $R$ in $\OreGen$ is also commutative, hence a maximal commutative subring of $\OreGen$. Indeed, take two arbitrary elements $\sum_{i=0}^n c_i x^i$ and $\sum_{j=0}^m d_j x^j$
in the centralizer of $R$ and compute
\begin{align*}
 \left( \sum_{i=0}^n c_i x^i \right) \left( \sum_{j=0}^m d_j x^j \right) = \sum_{j=0}^m d_j \left( \sum_{i=0}^n c_i x^i \right) x^j = \sum_{j=0}^m \sum_{i=0}^n d_j c_i x^{i+j} =\\
\sum_{i=0}^n \sum_{j=0}^m c_i d_j x^j x^i = \sum_{i=0}^n c_i \left( \sum_{j=0}^m d_j x^j \right) x^i = \left( \sum_{j=0}^m d_j x^j\right ) \left( \sum_{i=0}^n c_i x^i \right). 
\end{align*}

\end{remark}

\begin{prop}
\label{MaxCommWithDelta}
Let $R$ be a commutative ring.
If for every $n \in \mathbb{Z}_{>0}$ there is some $r \in R$ such that $\sigma^n(r)-r$ is a regular element,
then $R$ is a maximal commutative subring of $\OreGen$.
In particular, if $R$ is an commutative domain and $\sigma$ is of infinite order, then $R$ is maximal commutative.
\end{prop}

\begin{proof}
Suppose that $P = \sum_{k=0}^n a_k x^k$ is an element of degree $n>0$ which commutes with every element of $R$.
Let $r$ be an element of $R$ such that $\sigma^n(r)-r$ is regular.
By Proposition~\ref{CommutantR} and the commutativity of $R$, we get that
$r a_n = \sigma^n(r) a_n$ or equivalently $(\sigma^n(r)-r)a_n=0$.
Since $\sigma^n(r)-r$ is regular this implies $a_n=0$, which is a contradiction.
This shows that $R$ is a maximal commutative subring of $\OreGen$.
\end{proof}

\begin{exmp}[The quantum Weyl algebra]\label{qWeyl}
Let $k$ be an arbitrary field of characteristic zero and let $R:=k[y]$ be the polynomial ring in one indeterminate over $k$.

Define $\sigma(y)=qy$ for some $q\in k \setminus \{0,1\}$. Then for any $p(y) \in k[y]$ we have $\sigma(p(y))=p(qy)$ and $\sigma$ is an automorphism of $R$.
Define a map $\delta : R \to R$ by 
\begin{displaymath}
\delta(p(y)) = \frac{\sigma(p(y))-p(y)}{\sigma(y)-y} = \frac{p(qy)-p(y)}{qy-y}                                    
\end{displaymath}
for $p(y) \in k[y]$. One easily checks that $\delta$ is a well-defined $\sigma$-derivation of $R$.
The ring $k[y][x;\sigma, \delta]$ is known as the \emph{$q$-Weyl algebra} or the \emph{$q$-deformed Heisenberg algebra} \cite{HSbook}. If $q$
is not a root of unity, then by Proposition~\ref{MaxCommWithDelta}, $k[y]$ is maximal commutative.
If $q$ is a root of unity of order $n$, then $x^n$ and $y^n$ are central and in particular $R$ is not maximal commutative.
\end{exmp}

\begin{remark}
Example~\ref{WeylAlgebra} demonstrates that infiniteness of the order of $\sigma$ in Proposition~\ref{MaxCommWithDelta} is not a necessary condition for $R$ to be a maximal commutative subring.
\end{remark}

\subsection{Skew polynomial rings}
\label{subsec:skewpolringscentermaxcom}
Many of the formulas simplify considerably if we take $\delta \equiv 0$, and as a consequence we can say more about maximal commutativity of $R$ in $R[x;\sigma,0]$.

\begin{prop} \label{MaxCommSkewPol}
Let $R$ be a commutative domain and $\SkewPol$ a skew polynomial ring.
$R$ is a maximal commutative subring of $\SkewPol$ if and only if
$\sigma$ is of infinite order.
\end{prop}

\begin{proof}
One direction is just a special case of Proposition~\ref{MaxCommWithDelta}. If $n \in \Z_{>0}$ is such that $\sigma^n=\identity_R$,
then the element $x^n$ commutes with each $r \in R$ since
$x^n r = \sigma^n(r) x^n = r x^n$.
\end{proof}

\begin{exmp}[The quantum plane]
With the same notation as in Example~\ref{qWeyl} form the ring $k[y][x;\sigma,0]$. It is known as the \emph{quantum plane}. By Proposition 
\ref{MaxCommSkewPol} $k[y]$ is a maximal commutative subring if and only if $\sigma$ is of infinite order , which is the same as saying
 that $q$ is not a root of unity. 
If $q$ is a root of unity of order $n$ then it is easy to see that $x^n$ and $y^n$ will belong to the center, hence $R$ is not a maximal commutative subring.
\end{exmp}

The following example shows that the conclusion of Proposition~\ref{MaxCommSkewPol}
is no longer valid if one removes the assumption that $R$ is a commutative domain.

\begin{exmp}\label{Ex_MaxComm}
Let $R$ be the ring $\mathbb{Q}^\mathbb{N}$ of functions from the non-negative integers to the rationals. Define $\sigma : R \to R$ such that, for any $f\in R$, we have $\sigma(f)(0)=f(0)$ and $\sigma(f)(n) = f(n-1)$ if $n>0$.
Then $\sigma$ is an injective endomorphism. But $d_0$, the characteristic function of $\{ 0 \}$, satisfies $d_0(n) (\sigma(f)(n)-f(n))=0$ for all $f \in R$ and $n \in \N$. Thus, as in the proof of 
Propostion \ref{MaxCommSkewPol}, it follows that the
element $d_0 x$ of $\SkewPol$ commutes with everything in $R$.
\end{exmp}

\subsection{Differential polynomial rings}
\label{subsec:Differentialpolringscentermaxcom}
We shall now direct our attention to the case when $\sigma=\identity_R$.
The following useful lemma appears in \cite[p. 27]{Goodearl2004}.

\begin{lemma}\label{CommutationRuleDiffRing}
In $\DiffPol$ we have
\begin{displaymath}
	x^n r = \sum_{i=0}^n \binom{n}{i} \delta^{n-i}(r) x^i
\end{displaymath}
for any non-negative integer $n$ and any $r\in R$.
\end{lemma}

We will make frequent reference to the following lemma, which is an easy consequence of Lemma \ref{CommutationRuleDiffRing}.

\begin{lemma} \label{commutantLemma} Let $q= \sum_{i=0}^n q_i x^i \in \DiffPol$ and $r \in Z(R)$. The following assertions hold: 
\begin{enumerate}[{\rm (i)}]
\item if $n=0$, then $rq-qr=0$; 
\item if $n \geq 1$, then $rq-qr$ has degree at most $n-1$ and $(n-1)$:th coefficient $-nq_n \delta(r)$; 
\item $xq-qx= \sum_{i=0}^n \delta(q_i) x^i$.
\end{enumerate}
\end{lemma}

The following proposition gives some sufficient conditions for $R$ to be a maximal commutative subring of $\DiffPol$. 
Note that in the special case when $R$ is commutative and $\sigma= \identity_R$, an outer derivation is the same as a non-zero derivation. 

\begin{prop}\label{MaxCommutativityDiffRing}
Let $R$ be a commutative domain of characteristic zero.
If the derivation $\delta$ is non-zero, then $R$ is a maximal commutative subring of $\DiffPol$.
\end{prop}

\begin{proof}
Suppose that $R$ is not a maximal commutative subring of $\DiffPol$.
We want to show that $\delta$ is zero.
By our assumption, there is some $n \in \Z_{>0}$ and some
$q = b x^n + a x^{n-1} + [\text{lower terms}]$ with $a,b \in R$ and $b\neq 0$ such that
$rq-qr=0$ for all $r\in R$.
By Lemma \ref{commutantLemma} and the commutativity of $R$, we get
$ rq -qr =  (-nb\delta(r))x^{n-1} + [\text{lower terms}].$

Hence, for any $r\in R$, $nb\delta(r)=0$ which yields $n\delta(r)=0$ since $R$ is a commutative domain
and $\delta(r)=0$ since $R$ is of characteristic zero.
\end{proof}

\begin{exmp}
 Let $k$ be a field of characteristic $p >0$ and let $R=k[y]$. If we take $\delta$ to be the usual formal derivative, then we note that $x^p$ is a central element in $\DiffPol$. This shows that the assumption on the 
characteristic of $R$ in Proposition~\ref{MaxCommutativityDiffRing} can not be relaxed.  
\end{exmp}

\section{Simplicity conditions for $\OreGen$}
\label{sec:simplicity}
Now we proceed to the  main topic of this article. We investigate when $\OreGen$ is simple and demonstrate how this is related to maximal commutativity of $R$ in $\OreGen$.

In any skew polynomial ring $R[x;\sigma,0]$, the ideal generated by $x$ is proper and hence skew polynomial rings can never be simple. In contrast, there exist simple skew Laurent rings (see e.g. \cite{JordanLaurent}).

\begin{remark}\label{innerderivationnonsimple}
If $\delta$ is an inner derivation, then $\OreGen$ is isomorphic to a skew polynomial ring and hence not simple (see
\cite[Lemma 1.5]{Goodearl92}).
\end{remark}

We are very interested in finding an answer to the following question.
\begin{question}\label{question_injective}
Let $R[x;\sigma,\delta]$ be a general Ore extension ring where $\sigma$ is, a priori, not necessarily injective.
Does the following implication always hold?
\begin{displaymath}
R[x;\sigma,\delta]  \text{ is a simple ring.} \Longrightarrow \sigma \text{ is injective}.
\end{displaymath}
\end{question}
So far, we have not been able to find an answer in the general situation. However, it is clear that the implication holds in the particular case when $\delta(\ker \sigma) \subseteq \ker \sigma$, for example when 
$\sigma$ and $\delta$ commute.
 
The following unpublished partial answer to the question has been communicated by Steven Deprez (see \cite{MathOverflowQuestion}).

\begin{prop}\label{Steven}
Let $R$ be a commutative and reduced ring. If $\OreGen$ is simple, then $\sigma$ is injective. 
\end{prop}

\begin{proof}
 Suppose that $\sigma$ is not injective. Take $a \in \ker(\sigma) \setminus \{ 0\}$. By assumption $a^k \neq 0$ for all $k \in \Z_{>0}$. 
Define $I= \{ p \in \OreGen \mid \exists k \in \Z_{>0} : pa^k =0 \}.$ It is clear that $I$ is a left ideal of $\OreGen$, and a right $R$-module. It is non-zero since it contains $ax-\delta(a)$. Since $a$ is not nilpotent, $I$ does 
not contain $1$. If we show that $I$ is closed under right multiplication by $x$, then we have shown that it is a non-trivial ideal of $\OreGen$. Take any $p \in I$ and $k$ such that $pa^k=0$. We compute
\begin{displaymath}
 (px)a^{k+1} = p(xa)a^k = p(\sigma(a)x+\delta(a))a^k = p\delta(a)a^k=0.  
\end{displaymath}
This shows that $px \in I$.
\end{proof}

Lemma 1.3 in \cite{Goodearl92} implies as a special case the following.
\begin{lemma}\label{SigmaDeltaExtension}
 If $R$ is a commutative domain, $k$ its field of fractions,  $\sigma$ an injective endomorphism of $R$ and $\delta$ a $\sigma$-derivation of $R$, then $\sigma$ and $\delta$ extends uniquely to $k$ as an injective endomorphism,
respectively a $\sigma$-derivation.
\end{lemma}

Using Proposition \ref{Steven} we are able to generalize a result proved by Bavula in \cite{Bavula}, using his technique. 

\begin{theorem}\label{SimpleOre}
If $R$ is a commutative domain and $\OreGen$ is a simple ring, then $\sigma =\identity_R$.
\end{theorem}

\begin{proof}
By Proposition \ref{Steven}, $\sigma$ must be injective. 

Let $k$ be the field of fractions of $R$. By Lemma~\ref{SigmaDeltaExtension}, $\sigma$ and $\delta$ extend uniquely to $k$. $\OreGen$ can be seen as a subring of $k[x;\sigma,\delta]$. If $\sigma \neq \identity_R$, 
then there is some
$\alpha \in R$ such that $\sigma(\alpha)-\alpha \neq 0$.
For every $\beta \in k$ we have $\delta(\alpha \beta) = \delta(\beta \alpha)$. Hence, for every $\beta \in k$ the following three equivalent identities hold.  
\begin{eqnarray*}
 \sigma(\alpha) \delta(\beta) +\delta(\alpha)\beta = \sigma(\beta)\delta(\alpha)+\delta(\beta)\alpha \Leftrightarrow (\sigma(\alpha)-\alpha)\delta(\beta) = (\sigma(\beta)-\beta)\delta(\alpha) \\
 \Leftrightarrow \delta(\beta) = \frac{\delta(\alpha)}{\sigma(\alpha)-\alpha}(\sigma(\beta)-\beta).
\end{eqnarray*}
Hence $\delta$ is an inner $\sigma$-derivation. This implies that $k[x;\sigma,\delta]$ is not simple since it is isomorphic to a skew polynomial ring.
Letting $I$ be a proper ideal of $k[x;\sigma,\delta]$ one can easily check that $I \cap \OreGen$ is a proper ideal of $\OreGen$, which is a
contradiction.
\end{proof}

An example in \cite[Chapter 3]{CozzensFaith} shows that Theorem \ref{SimpleOre} need not hold if $R$ is a non-commutative domain.

\begin{definition}
An ideal $J$ of $R$ is said to be \emph{$\sigma$-$\delta$-invariant} if $\sigma(J)\subseteq J$ and $\delta(J)\subseteq J$.
If $\{0\}$ and $R$ are the only $\sigma$-$\delta$-invariant ideals of $R$, then $R$ is said to be \emph{$\sigma$-$\delta$-simple}.
\end{definition}

The following necessary condition for $\OreGen$ to be simple is presumably well-known but we have not been able to find it in the existing literature. For the convenience of the reader, we provide a proof.

\begin{prop}\label{SimpleInvariant}
If $\OreGen$ is simple, then $R$ is $\sigma$-$\delta$-simple.
\end{prop}

\begin{proof}
Suppose that $R$ is not $\sigma$-$\delta$-simple and let $J$ be a non-trivial
$\sigma$-$\delta$-invariant ideal of $R$. Let $A=\OreGen$.
Consider the set $I=JA$ consisting of finite sums of elements of the form
$ja$ where $j\in J$ and $a\in A$. We claim that $I$ is a non-trivial ideal of $A$,
and therefore $\OreGen$ is not simple;

Indeed, $I$ is clearly a right ideal of $A$, but it is also a left ideal of $A$.
To see this, note that
for any $r\in R$, $j\in J$ and $a\in A$ we have $rja \in I$
and by the $\sigma$-$\delta$-invariance of $J$ we conclude that
$xja = \sigma(j)xa+\delta(j)a \in I$.
By repeating this argument we conclude that $I$ is a two-sided ideal of $A$.
Furthermore, $I$ is non-zero, since $A$ is unital and $J$ is non-zero,
and it is proper; otherwise we would have
$1=\sum_{i=0}^n j_i a_i$ for some $n\in \Z_{\geq 0}$, $j_i \in J$ and $a_i \in R$ for $i\in\{0,\ldots,n\}$,
which implies that $1 \in J$ and this is a contradiction.
\end{proof}

\begin{exmp}
One can always, even if $\OreGen$ is simple, find a (non-zero) left ideal $I$ of $\OreGen$ such that $I\cap R = \{0\}$.
Take some $n \in \Z_{>0}$ and let $I$ be the left ideal generated by $1-x^n$. This left ideal clearly has the desired property.
\end{exmp}

\begin{remark}\label{CenterField}
Recall that the center of a simple ring is a field.
\end{remark}

\begin{lemma}\label{CenterIntersection2}
 If $R$ is a domain and $\delta$ is non-zero, then $r \in R$ belongs to $Z(\OreGen)$ if and only if the following two assertions hold:

\begin{enumerate}[{\rm (i)}]
    \item $\delta(r)=0$;
    \item $r \in Z(R)$.
\end{enumerate}
\end{lemma}

\begin{proof}
Since every element of $Z(\OreGen)$ commutes with $x$ and with each $r' \in R$, it is clear that conditions (i) and (ii) are necessary for $r \in R$ to belong to $Z(\OreGen)$. 
We also see that they are sufficient if they imply that $\sigma(r)=r$. 

Now, suppose that (i) and (ii) hold. Since $\delta$ is non-zero there is some $b$ such that $\delta(b) \neq 0$. We compute $\delta(rb)$ and $\delta(br)$ which must be equal since $r \in Z(R)$. A calculation yields
\begin{align*}
 \delta(br) =& \sigma(b) \delta(r)+\delta(b) r = r \delta(b), \\
 \delta(rb) =& \sigma(r) \delta(b) +\delta(r) b = \sigma(r) \delta(b).
\end{align*}
So $(\sigma(r)-r)\delta(b) =0$. This implies that $\sigma(r)=r$.
\end{proof}

\begin{prop}\label{SimpleCenter}
Let $R$ be a domain  and $\sigma$ injective. The following assertions hold:
\begin{enumerate}[{\rm (i)}]
	\item $\OreGen \setminus R$ contains no invertible element;
	\item if $\OreGen$ is simple, then the center of $\OreGen$ is contained in $R$
	and consists of those $r \in Z(R)$ such that $\delta(r)=0$.
\end{enumerate}
\end{prop}

\begin{proof}
(i) This follows by the same argument as in \cite[Theorem 1.2.9(i)]{McConnellRobson}.
(ii) This follows from (i), Remark \ref{CenterField} and Lemma \ref{CenterIntersection2} since $\delta$ must be non-zero.
\end{proof}

\subsection{Differential polynomial rings}
\label{subsec:simplicityDifferentialPolrings}
We shall now focus on the case when $\sigma = \identity_R$.

Note that for a derivation $\delta$ on $R$ we have the \emph{Leibniz rule}:
\begin{displaymath}
	\delta^n(rs)=\sum_{i=0}^n \binom{n}{i} \delta^{n-i}(r) \delta^i(s)
\end{displaymath}
for $n\in \Z_{\geq 0}$ and $r,s\in R$.

\begin{prop}\label{IdealIntersection}
$I \cap Z(R)' \neq \{ 0\}$ holds for any non-zero ideal $I$ of $\DiffPol$, where $Z(R)'$ denotes the centralizer of $Z(R)$ in $\DiffPol$.
\end{prop}

\begin{proof}
Let $I$ be an arbitrary non-zero ideal of $\DiffPol$.
Take $a \in I \setminus \{0\}$ such that $n:=\deg(a)$ is minimal. If $n=0$ then we are done. 
Otherwise, if $a$ is of degree $n>0$ it follows from Lemma \ref{commutantLemma}(ii) that 
 $\deg(ra-ar) < \deg(a)$. Since $ra-ar \in I$ we conclude by the minimality of $\deg(a)$ that $ra-ar=0$.   
Hence $I \cap Z(R)' \neq \{0\}$.
\end{proof}

\begin{corollary}\label{MaxCommIdealIntersection}
If $R$ is a maximal commutative subring of $\DiffPol$,
then $I \cap R \neq \{0\}$ holds for any non-zero ideal $I$ of $\DiffPol$.
\end{corollary}

We have seen that if $\DiffPol$ is a simple ring, then its center is a field and $R$ is $\delta$-simple. These necessary conditions are well-known, see e.g. \cite{Goodearl2004}. 
We will now show that they are also sufficient and begin with the following lemma.

\begin{lemma}\label{monic}
 Let $S=\DiffPol$ be a differential polynomial ring where $R$ is $\delta$-simple. For every element $b \in S \setminus \{0 \}$ we can find an element $b' \in S$ such that: 
\begin{enumerate}[{\rm (i)}]
 \item $ b' \in SbS$;
 \item $\deg(b') = \deg(b)$; 
 \item $b'$ has $1$ as its highest degree coefficient.
\end{enumerate}
\end{lemma}

\begin{proof}
Let $J$ be an arbitrary ideal of $\DiffPol$ and $n$ an arbitrary non-negative integer. 
  Define the following set
\begin{displaymath}
 H_n(J) = \{ a \in R \ | \  \exists \, c_0,c_1,\cdots c_{n-1} \in R : \ ax^n + \sum_{i=0}^{n-1} c_i x^i \in J \},
\end{displaymath}
consisting of the $n$:th degree coefficients of all elements in $J$ of degree \emph{at most} $n$. 

Clearly, $H_n(J)$ is an additive subgroup of $R$. Take any $r \in R$. If $ax^n +\sum_{i=0}^{n-1} c_i x^i$ belongs to $J$, then so does $ra x^n +\sum_{i=0}^{n-1} r c_i x^{i}$. Thus, $H_n(J)$ is a left ideal of $R$. Furthermore, if 
$c = ax^n +\sum_{i=0}^{n-1} c_i x^i$ is an element of $J$ then so is $cr$, and it is not difficult to see that $cr$ has degree at most $n$ and that its $n$:th degree coefficient is $ar$. Thus, $H_n(J)$ is also a right ideal of $R$
and hence an ideal.

We claim that $H_n(J)$ is a $\delta$-invariant ideal. Indeed, take any $a \in H_n(J)$ and a corresponding element $ax^n + \sum_{i=0}^{n-1} c_i x^i \in J$. Then we get
\begin{align*}
  x(ax^n + \sum_{i=0}^{n-1} c_i x^i) -(ax^n + \sum_{i=0}^{n-1} c_i x^i)x 
   = \delta(a) x^n + \sum_{i=0}^{n-1} \delta(c_i) x^i  \in J.
\end{align*}
This implies that $\delta(a) \in H_n(J)$ and that $H_n(J) $ is $\delta$-invariant.

Now, take any $b\in S \setminus \{0\}$ and put $n=\deg(b)$.
Let $b_n$ denote the $n$:th degree coefficient of $b$.
Put $J=SbS$ and note that $b_n \in H_n(SbS)$.
Since $R$ is $\delta$-simple and $H_n(SbS)$ is non-zero we conclude that $H_n(SbS)=R$. Thus, $1 \in H_n(SbS)$ and the proof is finished.  
\end{proof}

We now show that the assumption that $R$ is $\delta$-simple allows us to reach a stronger conclusion than in Proposition \ref{IdealIntersection}.

\begin{prop}\label{CenterIdealIntersection}
Let $S=\DiffPol$ be a differential polynomial ring where $R$ is $\delta$-simple. Then $I \cap Z(S) \neq \{ 0\}$ holds for every non-zero ideal $I$ of $S$. 
\end{prop}

\begin{proof}
Let $I$ be any non-zero ideal of $S$ and choose a non-zero element $b \in I$ of minimal degree $n$. By Lemma \ref{monic} we may assume that its highest degree coefficient is $1$. Let us write $b = x^n +\sum_{i=0}^{n-1} c_i x^i$.
Since $b_n=1$ we have that $\deg(rb-br) < \deg(b)$ for all $r$ in $R$.  Since $rb-br \in I$ it follows from the minimality of $\deg(b)$ that $rb-br=0$, i.e.  $b$ commutes with each element in $R$.

Similarly note that $xb-bx \in I$. By Lemma \ref{commutantLemma}(iii) and the fact that $\delta(1)=0$
we get $\deg(xb-bx) < \deg(b)$ which implies that $xb-bx=0$. Since $R$ and $x$ generate $S$, $b$ must lie in the center of $S$. 
\end{proof}

We now obtain the promised characterization of when $\DiffPol$ is simple. In \cite{Hauger} Hauger obtains a similar result for a class of rings that are similar to, but distinct from, the ones studied in the present article. 
Hauger's method of proof is also different from ours.  

\begin{theorem}\label{MainResult2}
 Let $R$ be a ring and $\delta$ a derivation of $R$. The differential polynomial ring $\DiffPol$ is simple if and only if $R$ is $\delta$-simple and $Z(\DiffPol)$ is a field.
\end{theorem}

\begin{proof}
If $R$ is $\delta$-simple and $Z(\DiffPol)$ is a field then, by Proposition \ref{CenterIdealIntersection}, $\DiffPol$ is simple. The converse follows 
from Proposition \ref{SimpleInvariant} and Remark \ref{CenterField}. 
\end{proof}

A different sufficient condition for $\DiffPol$ to be simple is given by the following. 

\begin{prop}\label{DeltaMaxCommSimple}
If $R$ is $\delta$-simple and a maximal commutative subring of $\DiffPol$,
then $\DiffPol$ is a simple ring.
\end{prop}

\begin{proof}
Let $J$ be an arbitrary non-zero ideal of $\DiffPol$. Using the notation of the proof of Lemma \ref{monic} we see that $H_0(J)=J \cap R$. 
By Corollary \ref{MaxCommIdealIntersection} and the proof of Lemma \ref{monic}, it follows that $H_0(J)$ is a non-zero $\delta$-invariant ideal of $R$.
By the assumptions we get $H_0(J)=R$, which shows that $1_{\DiffPol} \in J$.
Thus, $J=\DiffPol$.
\end{proof}

\begin{remark}
 Under the assumptions of Proposition \ref{DeltaMaxCommSimple}, the center of $\DiffPol$ consists of the constants in $R$. 
\end{remark}

In the following example we verify the well-known fact that the Weyl algebra is simple as an application of Proposition \ref{DeltaMaxCommSimple}. 

\begin{exmp}[The Weyl algebra]\label{WeylAlgebra}
 Take $R=k[y]$ for some field $k$ of characteristic zero. Let $\sigma=\identity_R$ and define $\delta$ to be the usual formal derivative of polynomials. Then $R[x;\sigma,\delta]$ is the Weyl algebra.
It is easy to see that $k[y]$ is $\delta$-simple and a maximal commutative subring of the Weyl algebra and thus $\OreGen$ is simple by Proposition~\ref{DeltaMaxCommSimple}. 
\end{exmp}

Maximal commutativity of $R$ in $\DiffPol$ does not imply $\delta$-simplicity of $R$, as demonstrated by the following example.  

\begin{exmp}
Let $k$ be a field of characteristic zero and take $R=k[y]$. Define $\delta$ to be the unique derivation on $R$ satisfying $\delta(y)=y$ and $\delta(c)=0$ for $c\in k$. No element outside of $R$ commutes with $y$ and thus
 $R$ is a maximal commutative subring of $\DiffPol$.   
However, $R$ is not $\delta$-simple since $Ry$ is a proper $\delta$-invariant ideal of $R$. 
\end{exmp}

The following lemma appears as \cite[Lemma 4.1.3]{Jordan} and also follows from Proposition \ref{SimpleInvariant} and Remark \ref{innerderivationnonsimple}.
\begin{lemma}\label{SimpleDiffRing}
If $\DiffPol$ is simple, then $R$ is $\delta$-simple and $\delta$ is outer.
\end{lemma}

\begin{corollary}\label{cor_deltasimple2}
Let $R$ be a commutative domain of characteristic zero.
If $\DiffPol$ is simple, then $R$ is a maximal commutative subring of $\DiffPol$.
\end{corollary}

\begin{proof}
This follows from Lemma \ref{SimpleDiffRing} and Proposition \ref{MaxCommutativityDiffRing}. 
\end{proof}

The following proposition follows from \cite[Theorem 4.1.4]{Jordan}.   

\begin{prop}\label{EquivOreGen}
Let $R$ be a commutative ring that is torsion-free as a module over $\Z$.
The following assertions are equivalent:
\begin{enumerate}[{\rm (i)}]
	\item $R[x;\identity_R, \delta]$ is a simple ring;
	\item $R$ is $\delta$-simple and $\delta$ is non-zero.
	\end{enumerate}
\end{prop}

\begin{proof}
(i)$\Rightarrow$(ii): This follows from Lemma \ref{SimpleDiffRing}.\\
(ii)$\Rightarrow$(i): Suppose that $R$ is $\delta$-simple and $\delta$ is non-zero.
Let $J$ be an arbitrary non-zero ideal of $R[x;\identity_R, \delta]$.
Choose some $q\in J\setminus \{0\}$ of lowest possible degree, which we denote by $n$.
Seeking a contradiction, suppose that $n>0$.
By Lemma \ref{monic} we may assume that $q$ has $1$ as its highest degree coefficient.


Let $r \in R$ be arbitrary.
Lemma \ref{commutantLemma}(ii) yields $rq -qr = -n\delta(r)x^{n-1} + [\text{lower terms}]$.
By minimality of $n$ and the fact that $rq-qr \in I$, we get $rq-qr=0$.
Since $R$ is torsion-free, we conclude that $\delta(r)=0$. This is a contradiction
and hence $n=0$.
Thus, $q=1$ and hence $J=R[x;\identity_R,\delta]$.
\end{proof}

Example~\ref{ex_EquivOreGen} demonstrates that assertion (ii) in 
Proposition~\ref{EquivOreGen} does not imply assertion (i) for a general commutative ring $R$.

\begin{exmp}\label{ex_EquivOreGen}
 Let $\mathbb{F}_2$ be the field with two elements and put $R= \mathbb{F}_2 [y]/\langle  y^2\rangle $. The ideal of $\mathbb{F}_2 [y]$  generated by $y^2$  is invariant under $\frac{d}{dy}$.  Thus, $\frac{d}{dy}$
 induces a derivation $\delta$ on $R$ such that
$\delta(y)=1$. 

$R$ is clearly $\delta$-simple but $\DiffPol$ is not simple. To see this note that $x^2$ is a central element. From that it is easy to see that the ideal generated by $x^2$ is proper.
\end{exmp}

We are now ready to state and prove one of the main results of this article. Note that by Theorem \ref{SimpleOre} all simple Ore extensions over commutative domains of characteristic zero are differential polynomial rings. 

\begin{theorem}\label{MainResult}
Let $R$ be a commutative domain of characteristic zero. 
The following assertions are equivalent:
\begin{enumerate}[{\rm (i)}]
	\item $R[x;\identity_R, \delta]$ is a simple ring;
	\item $R$ is $\delta$-simple and a maximal commutative subring of $R[x;\identity_R, \delta]$.
\end{enumerate}
\end{theorem}

\begin{proof}
(i)$\Rightarrow$(ii): By Lemma \ref{SimpleDiffRing} $\delta$ is non-zero and $R$ is $\delta$-simple.
The result now follows from Proposition \ref{MaxCommutativityDiffRing}.\\
(ii)$\Rightarrow$(i): This follows from Proposition \ref{DeltaMaxCommSimple}.
\end{proof}

\section*{Acknowledgements}
The first author was partially supported by 
The Swedish Research Council (postdoctoral fellowship no. 2010-918) and The Danish National Research Foundation (DNRF) through 
the Centre for Symmetry and Deformation. The first and second authors were both supported by the Mathematisches Forschungsinstitut Oberwolfach as Oberwolfach Leibniz graduate students. 

This research was also supported in part by
the Swedish Foundation for International Cooperation in Research and Higher Education (STINT),
The Swedish Research Council (grant no. 2007-6338), The Royal Swedish Academy of Sciences, The Crafoord
Foundation and the NordForsk Research Network ``Operator Algebras and Dynamics'' (grant no. 11580). The authors are extremely grateful to David A. Jordan for making the effort to scan 
\cite{Jordan} and kindly providing them with a PDF-copy. The authors are also grateful to Steven Deprez, Ken Goodearl and Anna Torstensson for comments on some of the problems and results
considered in this paper.
The authors would like to thank an anonymous referee for making several helpful comments that lead to improvements of this article.

\bibliographystyle{amsalpha}

\begin{thebibliography}{99}
\bibitem{ArchbordSpielberg}
Archbold, R. J.,  Spielberg, J. S., Topologically free actions and ideals in discrete $C^{*}$-dynamical systems,
Proc. Edinburgh Math. Soc. (2) {\bf 37}, no. 1, 119--124 (1994)

\bibitem{Bavula}
 Bavula, V.,
The simple modules of the Ore extensions with coefficients from a Dedekind ring.
Comm. Algebra 27 (1999), no. 6, 2665--2699.

\bibitem{CarlsenSS}
Carlsen, T. M., Silvestrov, S.,
On the Exel crossed product of topological covering maps,
Acta Appl. Math. \textbf{108},  no. 3, 573--583 (2009)

\bibitem{CozzensFaith}
Cozzens, J., Faith, C.,
Simple Noetherian rings.
Cambridge Tracts in Mathematics, No. 69. Cambridge University Press, Cambridge-New York-Melbourne, 1975. xvii+135 pp.


\bibitem{DutkayJorg3}
Dutkay, D. E., Jorgensen, P. E. T., Martingales, endomorphisms, and covariant systems of operators in Hilbert space, J. Operator Theory {\bf 58}, no. 2, 269--310 (2007)

\bibitem{DutkayJorg4}
Dutkay, D. E., Jorgensen, P. E. T., Silvestrov, S.,
Decomposition of wavelet representations and Martin boundaries,
J. Funct. Anal. \textbf{262}, no. 3, 1043--1061 (2012)

\bibitem{MR2195591ExelVershik}
Exel, R., Vershik, A., $C^*$-algebras of irreversible dynamical systems, Canad. J. Math. {\bf 58}, no. 1, 39--63 (2006)


\bibitem{Goodearl82}
Goodearl, K.R., Warfield, R.B., Primitivity in differential operator rings, Math. Z.,  \textbf{180}, no. 4, 503--523 (1982)

\bibitem{Goodearl92}
Goodearl, K. R.,  Prime Ideals in Skew Polynomial Rings and Quantized Weyl Algebras,
Journal of Algebra 150 (1992), 324-377

\bibitem{Goodearl2004}
 Goodearl, K. R.,  Warfield, R. B.,
An introduction to noncommutative Noetherian rings.
Second edition. London Mathematical Society Student Texts, 61. Cambridge University Press, Cambridge, 2004.

\bibitem{Hauger}
Hauger, G., Einfache Derivationspolynomringe, Arch. Math. (Basel),\textbf{29}, no. 5, 491--496 (1977)

\bibitem{HSbook}
Hellstr\"om, L., Silvestrov, S., 
Commuting elements in q-deformed Heisenberg algebras. World Scientific Publishing Co., Inc., River Edge, NJ, 2000. 

\bibitem{Jacobson}
Jacobson, N., Pseudo-linear transformations, Annals of Math., \textbf{38}, 484--507 (1937)

\bibitem{JainLamLeroy}
Jain, S.K., Lam, T.Y., Leroy, A., Ore extensions and $V$-domains, Rings, Modules and Representations, Contemp. Math. Series AMS, 480, 263--288 (2009)

\bibitem{deJeu}
de Jeu, M., Svensson, C., Tomiyama, J.,
On the Banach $*$-algebra crossed product associated with a topological dynamical system,
J. Funct. Anal. \textbf{262}, no. 11, 4746--4765 (2012)

\bibitem{Jordan}
Jordan, D. A.,
Ore extensions and Jacobson rings, Ph.D. Thesis, University of Leeds, 1975

\bibitem{Jordan2}
Jordan, D. A.,
Primitive Ore extensions, Glasgow Math. J., \textbf{18}, no. 1, 93--97 (1977)

\bibitem{JordanLaurent}
Jordan, D. A., 
Simple skew Laurent polynomial rings, Comm. Algebra, \textbf{12}, no. 1-2, 135--137 (1984)


\bibitem{LamLeroy}
Lam, T. Y., Leroy, A.,
Homomorphisms between Ore extensions. Azumaya algebras, actions, and modules (Bloomington, IN, 1990), 83--110,
Contemp. Math., 124, Amer. Math. Soc., Providence, RI, 1992.

\bibitem{LundstromOinert}
Lundstr\"{o}m, P., \"{O}inert, J.,
Skew category algebras associated with partially defined dynamical systems,
Internat. J. Math. \textbf{23}, no. 4, pp. 16 (2012)

\bibitem{McConnellRobson}
McConnell, J. C., Robson, J. C.,
Noncommutative Noetherian rings,
Pure and Applied Mathematics (New York),  John Wiley  \& Sons, Ltd., Chichester, 1987.

\bibitem{McConnell}
McConnell, J. C., Sweedler, M. E., Simplicity of smash products, Proc. London Math. Soc. no. 3, \textbf{23}, 251--266, (1971)


\bibitem{MACbook3}
Mackey, G. W., Unitary Group Representations in Physics, Probability, and Number Theory, xxviii+402 pp.
Second edition. Advanced Book Classics. Addison-Wesley Publishing Company, Advanced Book Program, Redwood City (1989)


\bibitem{Malm}
Malm, D. R., Simplicity of partial and Schmidt differential operator
              rings, Pacific J. Math. \textbf{132}, no. 1, 85--112 (1988)



\bibitem{MathOverflowQuestion}
MathOverflow,
"Simple Ore extensions",  \\
http://mathoverflow.net/questions/68233/ 
Accessed on 2011-11-28.




\bibitem{OinertSimplegroupgradedrings}
{\"O}inert, J., Simple group graded rings and maximal commutativity,
Operator structures and dynamical systems, 159--175, Contemp. Math., 503,
Amer. Math. Soc., Providence, RI, 2009.
Öinert, Johan ;  Silvestrov, Sergei D.  Commutativity and ideals in pre-crystalline graded rings.
 Acta Appl. Math.  108  (2009),  no. 3, 603--615.
		

\bibitem{OinertLundstrom} {\"O}inert, J., Lundstr{\"o}m, P.,
Commutativity and ideals in category crossed products, Proc. Est. Acad. Sci. 59 (2010), no. 4, 338--346

\bibitem{OinertLundstrom2}
\"{O}inert, J., Lundstr\"{o}m, P.,
The Ideal Intersection Property for Groupoid Graded Rings,
Comm. Algebra \textbf{40}, no. 5, 1860-1871 (2012)

\bibitem{OinertSS1}
\"{O}inert, J., Silvestrov, S. D.,
Commutativity and ideals in algebraic crossed products,
J. Gen. Lie Theory Appl.  \textbf{2},  no. 4, 287--302 (2008)

\bibitem{OinertSS2}
\"{O}inert, J., Silvestrov, S. D.,
On a correspondence between ideals and commutativity in algebraic crossed products,
J. Gen. Lie Theory Appl.  \textbf{2},  no. 3, 216--220 (2008)

\bibitem{OinertSS3}
\"{O}inert, J., Silvestrov, S. D.,
Crossed product-like and pre-crystalline graded rings, in
\emph{Generalized Lie theory in mathematics, physics and beyond},
281--296, Springer, Berlin, (2009).

\bibitem{OinertSS4}
\"{O}inert, J., Silvestrov, S. D.,
Commutativity and ideals in pre-crystalline graded rings,
Acta Appl. Math. \textbf{108},  no. 3, 603--615 (2009)
		
\bibitem{OinertSS5}
\"{O}inert, J., Silvestrov, S., Theohari-Apostolidi, T., Vavatsoulas, H.,
Commutativity and ideals in strongly graded rings,
Acta Appl. Math.  \textbf{108},  no. 3, 585--602 (2009)

\bibitem{OstSam-book}
Ostrovsky\u{\i}, V., Samo\u{\i}lenko, Yu., Introduction to the Theory of Representations of Finitely Presented $*$-Algebras. I, iv+261 pp. Representations by bounded operators. Reviews in Mathematics and Mathematical Physics, {\bf 11}, pt.1. Harwood Academic Publishers, Amsterdam (1999)


\bibitem{Posner}
Posner, C. E., Differentiable Simple Rings, Proc. Amer, Math. Soc. \textbf{11}, 337--343 (1960)



\bibitem{RowenRing1}
Rowen, L. H.,
Ring theory. Vol. I. Pure and Applied Mathematics, 127. Academic Press, Inc., Boston, MA, 1988.

\bibitem{SSD1}
Svensson, C., Silvestrov, S., de Jeu, M., Dynamical Systems and Commutants in Crossed Products, Internat. J. Math. \textbf{18}, no. 4, 455--471 (2007)

\bibitem{SSD2}
Svensson, C., Silvestrov, S., de Jeu, M.,
Connections between dynamical systems and crossed products of Banach algebras by $\mathbb{Z}$, in \emph{Methods of spectral analysis in mathematical physics}, 391--401,
Oper. Theory Adv. Appl., 186, Birkhäuser Verlag, Basel, 2009. 

\bibitem{SSD3}
Svensson, C., Silvestrov, S., de Jeu, M.,
Dynamical systems associated with crossed products,
Acta Appl. Math. \textbf{108},  no. 3, 547--559 (2009)

\bibitem{SvT}
Svensson, C., Tomiyama, J.,
On the commutant of $C(X)$ in $C^*$-crossed products by $\mathbb{Z}$ and their representations,
J. Funct. Anal. 256 , no. 7, 2367--2386 (2009)



\bibitem{TomiyamaBook}
Tomiyama, J., Invitation to $C\sp *$-algebras and topological dynamics, x+167 pp. World Scientific Advanced Series in Dynamical Systems, 3. World Scientific, Singapore (1987)

\bibitem{Voskoglou}
Voskoglou, M. G., Simple Skew Polynomial Rings, Publ. Inst. Math. (Beograd) (N.S.), \textbf{37(51)}, 37--41 (1985)




\end{thebibliography}

\end{document}